\documentclass[12 pt]{amsart}

\usepackage{graphicx}                    
\usepackage{amsmath}                     
\usepackage{amsfonts}                    
\usepackage{amssymb}                     
\usepackage{amsthm}                      
\usepackage[normalem]{ulem}              
\usepackage[noadjust,verbose,sort]{cite} 
\usepackage{setspace}                    

\usepackage{stmaryrd}
\usepackage{color}
\newtheorem{theorem}{Theorem}
\usepackage{cite}

\newcommand{\ds}{\displaystyle}

\newcommand{\mbf}[1]{\mathbf{#1}}

\usepackage{imakeidx}
\usepackage{amsopn}

\usepackage{amssymb,amsmath,amsthm,amsfonts,amsopn,url,bbm, hyperref, comment}
\usepackage{enumitem}
\theoremstyle{plain}
\newtheorem{thm}{Theorem}[section]

\newtheorem{propo}[thm]{Proposition}
\newtheorem{lemma}[thm]{Lemma}
\newtheorem{cor}[thm]{Corollary}
\theoremstyle{definition}
\newtheorem{defn}[thm]{Definition}
\newtheorem*{defn*}{Definition}
\newtheorem{example}[thm]{Example}
\newtheorem*{example*}{Example}

\newtheorem{rmk}[thm]{Remark}
\newtheorem{ntn}[thm]{Notation}

\newcommand{\ideal}[1]{\mathfrak{#1}}
\newcommand{\m}{\ideal{m}}

\newcommand{\q}{\ideal{q}}

\DeclareMathOperator{\sdepth}{s \ depth}
\DeclareMathOperator{\cdepth}{c \ depth}
\DeclareMathOperator{\kdepth}{k \ depth}
\DeclareMathOperator{\rdepth}{r \ depth}
\DeclareMathOperator{\leftdepth}{left \ depth}
\DeclareMathOperator{\rightdepth}{right \ depth}
\DeclareMathOperator{\gr}{gr}
\DeclareMathOperator{\cgr}{c \ gr}
\DeclareMathOperator{\kgr}{k \ gr}
\DeclareMathOperator{\rgr}{r \ gr}

\DeclareMathOperator{\Ass}{Ass}
\DeclareMathOperator{\wAss}{\widetilde{Ass}}

\DeclareMathOperator{\Spec}{Spec}

\DeclareMathOperator{\depth}{depth}

\DeclareMathOperator{\Ext}{Ext}
\DeclareMathOperator{\Hom}{Hom}

\DeclareMathOperator{\sK}{sK}
\DeclareMathOperator{\K}{K}
\DeclareMathOperator{\ann}{ann}

\renewcommand{\phi}{\varphi}

\title[Depth and Associated Primes in the Perfect Closure $R^\infty$]{Generalized Depth and Associated Primes in the Perfect Closure $R^\infty$}
\author{George Whelan}
\address{George Whelan, Department of Arts and Sciences, George Mason University Korea, Incheon, Korea, $21985$}
\email{gwhelan@gmu.edu}
\date{}

\begin{document}
\maketitle

\begin{abstract}
For a reduced Noetherian ring $R$ of characteristic $p > 0$, in this paper we discuss an extension of $R$ called its perfect closure $R^\infty$.  This extension contains all $p^e$-th roots of elements of $R$, and is usually non-Noetherian.  We first define the generalized notions of associated primes of a module over a non-Noetherian ring.  Then for any $R$-module $M$, we state a correspondence between certain generalized prime ideals of $(R^\infty \otimes_R M)/N$ over $R^\infty$, and the union of associated prime ideals of $F^e(M)/N_e$ as $e \in \mathbb{N}$ varies.  Here $F$ refers to the Frobenius functor, and in the paper we define an $F$-sequence of submodules $\lbrace N_e \rbrace \subseteq \lbrace F^e(M) \rbrace$ as $e$ varies, while $\underrightarrow{\lim} \ N_e = N$.  Under the further assumptions that $M$ is finitely generated and $(R,\mathfrak{m})$ is an $F$-pure local ring, we then show that $\depth_R(F^e(M))$ is constant for $e \gg 0$, and we call this value the stabilizing depth, or $\sdepth_R(M)$.  Lastly, we turn to non-Noetherian measures of the depth of $R^\infty \otimes_R M$ over $R^\infty$, which generalize as well.  Two of these values are the $\kdepth$ and the $\cdepth$, and we show $\kdepth_{R^\infty} (R^\infty \otimes_R M) = \sdepth_R (M) \geq \cdepth_{R^\infty} (R^\infty \otimes_R M)$, while all three values are equal under certain assumptions.
\end{abstract}

\section{Introduction}

\indent \indent Let $R$ be a reduced commutative Noetherian ring of characteristic $p>0$ with unity.  Its perfect closure $R^\infty$ is the extension formed by adjoining all $p^e$-th roots of elements of $R$.  As we discuss in sections~\ref{sec:PerfClos} and~\ref{sec:FF} respectively, $R^\infty$ is constructed as a direct limit of extensions of $R$, and $R^\infty \otimes_R M$ is a direct limit of $F^e(M)$ for an $R$-module $M$, where $F$ denotes the Frobenius functor.  Over these iterations $F^e(M)$, it is natural to investigate what happens to fundamental notions such as their associated primes in $R$, and if $(R,\mathfrak{m})$ is local, the depth of $F^e(M)$.  However, if we compare any such findings to the same measures of $R^\infty \otimes_R M$ as an $R^\infty$-module, the investigation is immediately complicated by the fact that $R^\infty$ is almost always a non-Noetherian ring.  Both of these concepts are then not well defined because characterizations which are equivalent in the Noetherian context become distinct definitions.

\indent Over a non-Noetherian ring $S$, for a module $M$ the definition of associated primes of a module generalizes to the additional concepts of weakly associated primes, strong Krull primes, and Krull primes, which have the respective containments $\Ass_S(M) \subseteq \wAss_{S} (M) \subseteq \sK_{S} (M) \subseteq \K_{S} (M)$.  For an overview of this topic see~\cite{IrRu-ass}, and we define them in section~\ref{sec:GenAss}.

\indent For $R$ and $M$ as above, in order to state the behavior of associated prime ideals of $F^e(M)$ in full generality as $e$ grows, we must define in section~\ref{sec:F-seq} an $F$-sequence of sub-modules $\lbrace  N_e \rbrace \subseteq \lbrace F^e(M) \rbrace$.  The direct limit $N = \underrightarrow{\lim} \ N_e$ of such modules is a sub-module of $R^\infty \otimes_R M$.  This notion is a generalization of $f$-sequences of ideals as discussed in~\cite{NoSh.alpha}, and which we address in section~\ref{sec:f-Seq}.  Then using the fact that $\Spec(R)$ and $\Spec(R^\infty)$ are homeomorphic topological spaces under the contraction map $\phi : \Spec(R^\infty) \rightarrow \Spec(R)$, as shown in section~\ref{sec:PerfClos}, our main result in section~\ref{sec:GenAssResults} states that as $e$ varies:

\begin{align*}
\phi^{-1} \bigg( \bigcup \Ass_R \big( F^e(M)/N_e \big) \bigg) &= \wAss_{R^\infty} \big( (R^\infty \otimes_R M)/N \big) \\
&= \sK_{R^\infty} \big( (R^\infty \otimes_R M)/N \big).
\end{align*}

\noindent where $\lbrace N_e \rbrace$ is an $F$-sequence of sub-modules of $\lbrace F^e(M) \rbrace$.  An analogous result also holds for $f$-sequences of ideals.

\indent Regarding non-Noetherian depths, see~\cite{Ho-grade} and~\cite{Fb-ThyGr} for a discussion of these values.  We define them in section~\ref{sec:GenDepth}, and for a local ring $(S,\mathfrak{n})$ and $S$-module $M$, we have $\rdepth_{S} (M) \geq \kdepth_{S} (M) \geq \cdepth_{S} (M)$.  Again, these are all equal if $S$ is Noetherian and $M$ is finitely generated.

\indent Now letting $M$ be a finitely generated module over a Noetherian local $(R, \mathfrak{m})$, in general $\depth_R(M) \neq \depth_R \big( F^e(M) \big)$ as $e$ grows, as in example~\ref{example:DecrDepth}.  Under the assumption that $R$ is an $F$-pure ring, which we discuss in section~\ref{sec:F-pure}, we find that this value is non-increasing, and therefore must eventually be constant for $e \gg 0$.  In section~\ref{sec:SDepth} we call this value the the stabilizing depth:

\begin{center}
$\sdepth_R (M) := \depth_R \big( F^e(M) \big) \ \text{for } \ e \gg 0$.
\end{center} 

\indent Since this measure of $M$ is defined over iterations $F^e(M)$, in section~\ref{sec:DepthTypeComp} we compare it to the non-Noetherian depth measures of the limit $\underrightarrow{\lim} \ F^e(M) = R^\infty \otimes_R M$ as an $R^\infty$-module.  We find that in general:

\begin{center}
$\kdepth_{R^\infty} (R^\infty \otimes_R M) = \sdepth_R (M) \geq \cdepth_{R^\infty} (R^\infty \otimes_R M)$.
\end{center}

\indent Under the assumption that $R$ satisfies countable prime avoidance (see~\cite[Lemma 13.2]{LeWe-CMR}, \cite[Lemma 3]{BurchCPA}, \cite{RsPv-BCT}, \cite{HHexponent}), or if a specific application of the prime avoidance lemma is satisfied, we have:

\begin{center}
$\kdepth_{R^\infty} (R^\infty \otimes_R M) = \sdepth_R (M) = \cdepth_{R^\infty} (R^\infty \otimes_R M)$.
\end{center}

\subsection{Outline}

\indent \indent Section~\ref{sec:Prelims} consists of preliminary discussions of relevant topics to this paper: The perfect closure $R^\infty$, the Frobenius functor, Frobenius sequences of ideals and modules, non-Noetherian associated prime ideals, non-Noetherian depth, and $F$-pure rings. In section~\ref{sec:GenAssResults} we discuss our results regarding non-Noetherian associated prime ideals of $R^\infty \otimes_R M$ over $R^\infty$ for an $R$-module $M$.  In section~\ref{sec:F-pureResults} we assume $F$-purity for our rings $R$, we define stabilizing depth for a finitely generated $R$-module $M$, and we compare $\sdepth_R(M)$ to other non-Noetherian depth measures over $R^\infty$.

\section{Preliminaries}\label{sec:Prelims}

\subsection{Notation and conventions}\label{sec:NotCon}

\indent \indent All rings are commutative with unity, reduced, and have characteristic $p > 0$.  The reduced assumption for a ring $S$ means $\sqrt{0} = 0$ in $S$.  We write $q$ to refer to $p^e$ for $e \in \mathbb{N}$, and occasionally $q' = p^{e'}$.  $S$ will refer to an arbitrary ring, $R$ will refer to an arbitrary Noetherian ring, and if either is local we will respectively write $(S,\mathfrak{n})$ or $(R, \mathfrak{m})$.  If $\bf x$ is a finite sequence of elements in a ring $S$, $\bf|x|$ refers to the cardinality of the set $\lbrace \bf x \rbrace$.  For a ring $S$ the map $f:S \rightarrow S$ refers to the Frobenius endomorphism $f(s) = s^p$, while $f^e$ refers to the $e^{th}$ iteration of $f$ for $e \in \mathbb{N}$.  If $M$ is an $S$-module, $\Ass_S(M)$ refers to the associated prime ideals of $M$ in $\Spec(S)$, that is, prime ideals which are annihilators of elements of $M$.  If $I \subset S$ is an ideal, $I^{[q]} := \lbrace (x^q) \ | \ x \in I \rbrace$ is the $q$-th Frobenius power of $I$, $I^{[\frac{1}{q}]} := \lbrace (x^{\frac{1}{q}}) \ | \ x \in I, \ \text{and} \ x^{\frac{1}{q}} \in S \rbrace$, and $I^F$ refers to the Frobenius closure $I^F := \lbrace x \ | \ x^q \in I^{[q]} \ \text{for some} \ q \rbrace$ (alternately, $I^F := I R^\infty \cap R$ for $R^\infty$ defined below).  If $I=I^F$, it is Frobenius closed.

\subsection{The perfect closure of a reduced ring of characteristic $p > 0$}\label{sec:PerfClos}

\indent \indent A ring $S$ of characteristic $p > 0$ is \emph{perfect} if the Frobenius map is an automorphism on $S$.  Let $R$ be a reduced commutative Noetherian ring, whereby $f:R \rightarrow R$ is injective since there exists no non-zero $x \in R$ such that $x^p = 0$, but the map is not surjective in general.  The \emph{perfect closure} $R^\infty$ is an extension of $R$ which will contain all $q$-th roots of all elements of $R$.  Hence $f : R^\infty \rightarrow R^\infty$ will be an automorphism.  Greenberg showed that such a ring always exists, regardless of whether a ring is reduced~\cite{Gr-perfect}.  For an explicit construction which also addresses the non-reduced case see \cite{DJo-BiExt}, and a description in \cite{NoSh.alpha}.

\indent We include an explicit construction of $R^\infty$ using our reduced hypothesis for $R$.  Since $R$ is Noetherian and reduced, it has only finitely many associated prime ideals $P_1, \ldots P_n$, and $P_1 \cap \ldots \cap P_n = \sqrt{0} = 0$.  By virtue of $R$ being reduced, the natural map $R \rightarrow \prod_{i=1}^n(R/P_i)$ is injective.  Moreover, the total ring of quotients is $K = q(R) = \prod_{i=1}^n q(R/P_i)$.  For each $i$, let $K_i = q(R/P_i)$, the fraction field of $R/P_i$, and let $\overline{K}_i$ be its algebraic closure.  Let $\bar{K} = \prod_{i=1}^n \overline{K}_i$.  Let $R^{\frac{1}{q}} := \lbrace x \in \bar{K} \ | \ x^q \in R \rbrace$.  Then $R^\infty := \underrightarrow{\lim} \ R^{\frac{1}{q}} = \bigcup R^\frac{1}{q}$.

\indent $R^\infty$ will rarely be Noetherian.  For example if $R := \mathbb{F}_p[x]$, then $R^\infty = \mathbb{F}_p[x, x^\frac{1}{p}, x^\frac{1}{p^2}, \ldots]$.  Specifically, $R^\infty$ will be Noetherian if and only if $R$ is a direct product of finitely many fields \cite[Theorem 6.3]{NoSh.alpha}.  However $R$ and $R^\infty$ always share some key similarities which we will discuss throughout this paper.  

\indent In \cite[Theorem $6.1, \ i)$]{NoSh.alpha}, the authors showed that there exists an order isomorphism between $\Spec(R)$ and $\Spec(R^\infty)$.  That is, the contraction map $\phi :  \Spec(R^{\infty}) \rightarrow \Spec(R)$ is an order-preserving bijection.  We show that the two spectra actually share a stronger relationship in that they are homemorphic topological spaces.

\begin{theorem}
Let $R$ be a reduced Noetherian ring of characteristic $p > 0$, and let $R^\infty$ be its perfect closure.  Then the contraction map $ \phi : \Spec(R^\infty) \rightarrow \Spec(R)$ is a homeomorphism with respect to the Zariski topology.
\end{theorem}

\begin{proof} 
\indent $\phi$ is already known to be a bijection by \cite[Theorem $6.1, \ i)$]{NoSh.alpha}.  $\phi$ is always continuous, but we claim that $\phi^{-1} : \Spec(R) \rightarrow \Spec(R^\infty)$ is continuous as well.  Let $V(J) \subset \Spec(R^\infty)$ be a closed set for some ideal $J \subset R^\infty$.  We claim $\phi (V(J)) = V(J \cap R)$, and hence its pre-image under $\phi^{-1}$ is closed.  If $J \subseteq P^\infty$, then clearly $J \cap R \subseteq P^\infty \cap R = P$, whence  $P \in V(J \cap R)$.  Conversely, if $J \cap R \subseteq P$, fix $x^{\frac{1}{q}} \in J$.  Then $x \in (J^{[q]}R^\infty \cap R) \subseteq (J \cap R) \subseteq P = (P^\infty \cap R) \subseteq P^\infty$.  Hence $x^{\frac{1}{q}} \in P^\infty$, and $J \subseteq P^\infty$.
\end{proof}

\begin{ntn}
Due to this relationship, we will use the notation $P$ and $P^\infty$ to denote corresponding prime ideals in $R$ an $R^\infty$ respectively.
\end{ntn}

\subsection{The Frobenius functor}\label{sec:FF}

\indent \indent Let $R^{\frac{1}{f}}$ denote an $R$-$R$ bimodule, which is isomorphic to the ring $R^{\frac{1}{p}}$.  However, for $a^{\frac{1}{p}} \in R^{\frac{1}{f}}$ and $r,s \in R$, let the left and right actions by $R$ be $r \circ a^{\frac{1}{q}} \cdot s = r^{\frac{1}{q}} a^{\frac{1}{q}} s$.  Iterating, we have $R^{\frac{1}{f^e}} \cong R^{\frac{1}{q}}$, and left and right action by $R$ is $r \circ a^{\frac{1}{q}} \cdot s = r^{\frac{1}{q}} a^{\frac{1}{q}} s$.  $R^{\frac{1}{f^e}}$ is additionally a right $R^{\frac{1}{q'}}$-module for all $q' \leq q$ with action $a^{\frac{1}{q}} \cdot r^{\frac{1}{q'}} = a^{\frac{1}{q}} r^{\frac{1}{q'}}$, which is compatible with the right action by $R$ since $R \subseteq R^{\frac{1}{q}}$.

\begin{rmk}\label{rmk:stdconst}
We discuss the module $R^{\frac{1}{f}}$ rather than the standard $R^{f}$ or $R^F$ such as defined in \cite{BH} because under our construction the direct limit $\underrightarrow{\lim} \ F^e(M)$ will be an $R^\infty$-module.
\end{rmk}

\begin{defn}
 If $M$ is an $R$-module, the \emph{Frobenius functor} is a right exact functor in both the categories of right and left $R$-modules, defined by:

\begin{center}
$F(M) := R^{\frac{1}{f}} \otimes_R M$
\end{center}

\noindent and if $\phi : M \rightarrow N$ is an $R$-module map, then $F(\phi) : F(M)\rightarrow F(N)$ is the map:

\begin{center}
$F(\phi) := 1_{R^{1/f}} \otimes_R \phi$
\end{center}
\end{defn}

\begin{ntn} $m_e = \sum_{i=1}^n (s_i^{\frac{1}{q}} \otimes_R m_i)_e$ will refer to an arbitrary element of $F^e(M)$.  For example, let $\phi: M \rightarrow M$ be the identity map on an $R$-module $M$.  If $q < q'$, and if $\sum_{i=1}^n (s_i^{\frac{1}{q}} \otimes_R m_i)_e = m_e \in F^{e}(M)$, then its image under $F^{e'-e}(\phi)$ is $m_{e'} = \sum_{i=1}^n (s_i^{\frac{1}{q}} \otimes_R m_i)_{e'}$, where each $m_i \in M$ does not change, and each $s_i^{\frac{1}{q}} \in R^{\frac{1}{q}} \subseteq R^{\frac{1}{q'}}$ does not change.

\indent Similarly, $m_\infty$ will refer to elements of $R^\infty \otimes_R M$.
\end{ntn}

\indent $F(M)$ is an $R-R^{\frac{1}{p}}$ bi-module (and hence an $R-R$ bi-module) with left and right action (on a simple tensor):

\begin{align*}
r \ \circ \ m_1 \cdot s^\frac{1}{p} =& \ r \circ (a^{\frac{1}{p}} \otimes_R m)_1 \cdot s^{\frac{1}{p}} \\
= & \ (r \circ a^{\frac{1}{p}} \cdot s^{\frac{1}{p}} \otimes_R m)_1 \\
= & \ (r^{\frac{1}{p}} a^{\frac{1}{p}} s^{\frac{1}{p}} \otimes_R m)_1 \\
= & \ r^\frac{1}{p} m_1 s^\frac{1}{p} 
\end{align*}

\indent We now have a directed system $\big( F^e(M), (\psi_{e,e'} \otimes_R 1_M) \big)$, since we have $F^e(M) :=  (R^{\frac{1}{f^e}} \otimes_R M)$, where the maps $\psi_{e,e'} : R^{\frac{1}{f^{e}}} \hookrightarrow R^{\frac{1}{f^{e'}}}$ are embeddings.  Therefore,

\begin{center}
 $(R^\infty \otimes_R M) := \underrightarrow{\lim} \ F^e(M)$,
\end{center}

\noindent which is a right $R^\infty$-module (and hence a right $R \subseteq R^\infty$ module), whose action (on a simple tensor) is consistent with each $F^e(M)$: 

\begin{align*} 
m_\infty \cdot r^{\frac{1}{q'}} =& \ (a^{\frac{1}{q}} \otimes_R m)_{\infty} \cdot r^{\frac{1}{q'}} \\
=& \ (a^{\frac{1}{q}} \cdot r^{\frac{1}{q'}} \otimes_R m)_{\infty} \\ 
=& \ (a^{\frac{1}{q}} r^{\frac{1}{q'}} \otimes_R m)_{\infty} \\
=& \ m_\infty r^{\frac{1}{q'}} 
\end{align*}

\begin{rmk}\label{rmk:freakout}
Since the left action by $R$ on $F^e(M)$ is identical to simple left multiplication by the ring $R^{\frac{1}{q}}$, $F^e(M)$ will be finitely generated as a left $R$-module, even though it may not be finitely generated by right action , since $R$ may not be $F$-finite.
\end{rmk}

\begin{defn}
If $M, N$ are $R$-modules with $N \subseteq M$ the \emph{Frobenius closure} of $N$ in $M$ is defined $N^F_{M} := \lbrace m \in M \ | \ (1 \otimes_R m)_e \in N^{[q]} \subseteq F^e(M) \ \text{for} \ e \gg 0 \rbrace$.  Here $N^{[q]} := F^e(i)(N)$, where $i : N \rightarrow M$ is the embedding map, and $F^e(i)$ is its image under the $e$-th iteration of the Frobenius functor.
\end{defn}

\indent  Equivalently $N^F_{M} := \lbrace m \in M \ | \ \overline{(1 \otimes_R m)}_e = 0 \in F^e(M / N) \ \text{for} \ e \gg 0 \rbrace$.  Also note that in particular, $0^F_{M} := \lbrace m \in M \ | \ (1 \otimes_R m)_e = 0 \in F^e(M) \ \text{for} \ e \gg 0 \rbrace$.

\begin{rmk}\label{rmk:Zero}
\indent Note that since we are discussing a directed system, by definition $0 \in (R^\infty \otimes_R M)$ is precisely the image of the elements of $0^F_{F^e(M)}$ as $e$ varies, or $\underrightarrow{\lim} \ 0^F_{F^e(M)} = 0 \in R^\infty \otimes_R M$.  That is,

\begin{align*}
m_{e} \in 0^F_{F^e(M)} \ \text{for some} \ e & \Leftrightarrow m_{e'} = 0 \in F^{e'}(M) \ \text{for some} \ e \leq e' \\
& \Leftrightarrow m_{\infty} = 0 \ \text{in} \ R^\infty
\end{align*}

\end{rmk}

\subsection{Frobenius sequences}\label{sec:Frob-Seq}

\subsubsection{$f$-sequences of ideals}\label{sec:f-Seq}

\begin{defn}
In a reduced ring $S$ of characteristic $p > 0$, an \emph{$f$-sequence of ideals in $S$} is a descending chain of ideals $\lbrace J_e \rbrace$ such that $f^{-1}(J_{e+1}) = J_e$ for every $e$.  
\end{defn}

\begin{rmk}
We do not require for $S$ to be reduced for this definition, because as mentioned in~\cite[Remark $4.2, (ii)$]{NoSh.alpha}, $\sqrt{0} \subseteq J_e$ for all $e$.  However all rings in this paper are reduced, so we include this hypothesis in our definition.
\end{rmk}

\indent For an extensive discussion about the properties of $f$-sequences, see \cite{NoSh.alpha}.  We will not give the subject a similarly thorough treatment, however there are some features relevant to this paper.
  
\begin{rmk}\label{rmk:fseqfacts}
Let $S$ be a ring of characteristic $p > 0$, and let $\lbrace J_e \rbrace$ be an $f$-sequence in $S$.
\begin{enumerate}
\item For all $e$, $\Ass_S(S/J_e) \subseteq $ $\Ass_S(S/J_{e + 1})$ \cite[Remark $4.2, \ iv)$]{NoSh.alpha}.
\item  For all $e$, $J_e$ is Frobenius closed \cite[Remark $4.1 \ ii)$]{NoSh.alpha}.  Hence given an ideal $I \subset S$ we can view $\lbrace (I^{[q]})^F \rbrace$ as a minimal $f$-sequence, since any $f$-sequence which contians $I$ at some stage must therefore contain the entire sequence. 
\item $\sqrt{J_e} = \sqrt{J_{e'}}$ for all $e, e'$, or all ideals in the sequence have the same radical \cite[Remark $4.2, i)$]{NoSh.alpha}.
\item  For a Notherian ring $R$ there exists a bijective correspondence between $f$-sequences of $R$ and ideals in $R^\infty$ \cite[Corollary $3.2$]{NoSh.alpha}.  For an ideal $J \subset R^\infty$, this correspondence is given explicitly by $\Gamma \ : \ J \mapsto \lbrace J_e \rbrace$, where for every $e \in \mathbb{N}$, $J_e := \lbrace r \in R \ | \ r^{\frac{1}{q}} \in J \rbrace$.  Alternatively, we can say $J_e = f^e(J) \cap R$, and in particular note that $J_0 = J \cap R$.
\end{enumerate} 
\end{rmk}

\indent We include some examples of $f$-sequences.

\begin{example}\label{example:fseqs} Let $R$ be a Noetherian ring of characteristic $p > 0$.  The following sequences $\lbrace J_e \rbrace \subset R$ are $f$-sequences with corresponding ideals $J \subset R^\infty$.  In each example, for every $e \in \mathbb{N}$, let $J_e$ be as defined.
\begin{enumerate}
\item $ J_e  =  (I^{[q]})^F $, $J = I R^\infty$, where $I \subset R$ is any ideal. 
\item Letting  $R =  k[x,y]$, $ J_e  =  (x,y)^{[q]} = (x^q, y^q) $, $J = (x,y) R^\infty$, where $k$ is any field of characteristic $p > 0$.
\item Letting  $R =  k[x,y]$, $ J_e = (x, y^q) $, $J = (x, x^{\frac{1}{p}}, x^{\frac{1}{p^2}}, \ldots, y) R^\infty$.
\item $ J_e = P $ for some $P \in \Spec(R)$, $J = P^\infty$. 
\end{enumerate}
\end{example}

\indent Note that example $2$ is of the form $\lbrace I^{[q]} \rbrace$ for an $I \subset R$, but examples $1, 3$ and $4$ show that such a form does not characterize all $f$-sequences.  In general a sequence of Frobenius iterates $\lbrace I^{[q]} \rbrace$ for some $I \subset R$ will not yield an $f$-sequence since $I^{[q]}$ may not be Frobenius closed as $q$ varies.  However if $R$ is an $F$-pure ring, as we discuss below, all ideals are Frobenius closed, and hence $\lbrace I^{[q]} \rbrace$ will always be an $f$-sequence.

\indent Continuing with our assumption that $R$ is Noetherian, let $\lbrace J_e \rbrace$ be an $f$-sequence.  Fix $e$ and we have an ascending chain,

\begin{center}
$\ldots J_{e - 2} = f^{-2}(J_{e}) \cap R \ \supseteq \ J_{e - 1} = f^{-1}(J_{e}) \cap R \ \supseteq \  J_{e}$
\end{center}

\noindent which must therefore stabilize, usually after the zero-th term in the sequence $J_0$.  Suppose the chain stabilizes to some ideal $J$ for which $J = f^{-1}(J)$.  Unsurprisingly, $J$ is the radical ideal from \ref{rmk:fseqfacts}, fact $3$ above.

\begin{propo}
Let $R$ be a reduced Noetherian ring of characteristic $p > 0$, and Let $\lbrace J_e \rbrace_{e \in \mathbb{N}}$ be an $f$ sequence in $R$, and let let $\mathfrak{a} = \sqrt{J_e}$ for all $e$.  Let $J \subset R$ be the ideal such that the ascending chain

\begin{center}
$\ldots \supseteq J_{e-2} \supseteq J_{e - 1} \supseteq J_{e} \supset \ldots$
\end{center}

\noindent stabilizes, for $e \in \mathbb{N}$.  Then $J = \mathfrak{a}$.

\end{propo}

\begin{proof}
To show $J \subseteq \mathfrak{a}$, fix $r \in J$.  Then for all $e$, there exists a $q' \geq q$ such that $r^{q'} \in J_e$, and $r \in \sqrt{J_e} = \mathfrak{a}$. 

\indent Conversely if $r \in \mathfrak{a}$, then for all $e$ there exists an $n$ such that $r^n \in J_e$.  Let $q' > n$, and therefore $r^{q'} \in J_e$, whereby $f^{-e'}(r^{q'}) = r \in J$. 
\end{proof}

\begin{rmk}\label{shift}
Let $R$ be a reduced Noetherian ring of characteristic $p > 0$, and Let $\lbrace J_e \rbrace_{e \in \mathbb{N}}$ be an $f$ sequence in $R$.  Then there exists some $n$ such that $f^{-n}(J_0)$ is the radical ideal for the sequence.  Hence any $f$-sequence can be uniquely extended upward, or shifted to a longer $f$-sequence $\mathfrak{a}_e$ where $\mathfrak{a}_0 = \sqrt{J_e}$ for all $e$, $\mathfrak{a}_e := f^{-n}(J_e)$.
\end{rmk}

\subsubsection{$F$-sequences of modules}\label{sec:F-seq}

\indent \indent For a module $M$ over a Noetherian ring $R$ of characteristic $p>0$, for every $e$ we have a natural map $g_{e} : F^e(M) \rightarrow F^{e+1}(M)$, where $g_{e} := (\psi_{e,e+1} \ \otimes_R \ 1_M)$ as discussed in section~\ref{sec:FF}.  In particular $g_0 : M \rightarrow F(M)$, and in general for $e \leq e'$, we write $g_{e,e'} : F^e(M) \rightarrow F^{e'}(M)$.

\begin{ntn}
For $m_e \in F^e(M)$, we will write $g_{e}(m_{e}) = m_{e+1}$, and for $r \in R$, $g_{e}(r \circ m_{e}) = g_{e}(r^\frac{1}{q} m_{e}) = r^\frac{1}{q} m_{e+1} = r^p \circ m_{e+1}$.
\end{ntn}

\begin{defn} 
An \emph{$F$-sequence of sub-modules} of $F^e(M)$ for some $R$-module $M$ is a sequence of sub-modules $\lbrace N_e \rbrace \subseteq \lbrace F^e(M) \rbrace$ such that $(g_{e})^{-1}(N_{e+1}) = N_e$ for all $e$. 
\end{defn}

\begin{rmk}
Under the standard $R-R$ bi-module construction $R^{f^e}$ as mentioned in remark \ref{rmk:stdconst}, for every $e$, $g_e = f$ is the Frobenius map, and each $N_e$ is a sub-module of $M$.
\end{rmk}
 
\indent $F$-sequences of sub-modules are up to isomorphism a generalization of $f$-sequences of ideals, where $M = R$.  As such, all features from remark \ref{rmk:fseqfacts} generalize as well.  Some proofs and comments justifying these facts follow from similar techniques used by \cite{NoSh.alpha}.  First we state a lemma which we use for fact $3$.

\begin{lemma}\label{lem:fseqann}
Let $R$ be Noetherian commutative ring of characteristic $p>0$, let $M$ be an $R$-module, and let $N_e$ be an $F$-sequence of submodules on $M$, and let $J_e = \ann_R \big( F^e(M)/N_e \big)$ for every $e$.  Then $\lbrace J_e \rbrace$ is an $f$-sequence of ideals in $R$.  
\end{lemma}

\begin{proof}
For every $e$, we want to show $f^{-1}(J_{e+1}) = J_e$.  Fix $r \in f^{-1}(J_{e+1})$, then $r^p \in J_{e+1}$.  Now fix $m_e \in F^e(M) \setminus N_e$.  By remark~\ref{rmk:Fseqfacts}, number $2$ below,  $N_e = N_e^F$ and hence $g_{e}(m_e) = m_{e+1} \notin N_{e+1}$.  Since $r^p \in J^{e+1}$, we have $r^p \circ m_{e+1} = r^\frac{p}{qp}m_{e+1} = r^\frac{1}{q}m_{e+1} \in N_{e+1}$.  But $r \circ m_e = r^\frac{1}{q}m_e$, and thus $r  \circ  m_e \in g_{e}^{-1}(r^\frac{1}{q}m_{e+1}) \subseteq N_e$.  Thus $r \in J_e = \ann_R \big( F^e(M)/N_e \big)$.  

\indent Conversely, first note that $F^e(M)$ is generated by elements of the form $(1 \otimes_R m)_e$ for some $m \in M$, and their images under $g_{e}$ generate $F^{e+1}(M)$.  

\indent Now fix $r \in J_e$, and we want to show $r^p \in J_{e+1}$.  Let $(1 \otimes_R m)_{e+1}$ be a generator of $F^{e+1}(M)$, and consider $r^p \circ (1 \otimes_R m)_{e+1} = (r^{\frac{1}{q}} \otimes_R m)_{e+1}$.  Now, any $(r^{\frac{1}{q}} \otimes_R m)_{e} \in g_e^{-1}\big( (r^{\frac{1}{q}} \otimes_R m)_{e+1} \big)$ will be in $N_e$ since it is equal to $r \circ (1 \otimes_R m)_e$.  Therefore, $r^p \circ (1 \otimes_R m)_{e+1} \in N_{e+1}$, and since this was an arbitrary generator, we have $r^p \circ F^{e+1}(M) \subseteq N_{e+1}$, and $r^p \in J_{e+1}$.
\end{proof}

\begin{rmk}\label{rmk:Fseqfacts}
Let $R$ be a Noetherian ring of characteristic $p > 0$, and let $\lbrace N_e \rbrace$ be an $f$-sequence in $F^e(M)$.
\begin{enumerate}
\item For all $e$, $\Ass_R \big( F^e(M)/N_e \big) \subseteq $ $\Ass_R \big( F^{e+1}(M)/N_{e+1} \big)$.

\begin{proof}
\indent Let $P \in \Ass_R \big( F^e(M)/N_e \big)$ with $P = (N_e :_R m_{e})$, where $g_{e}(m_{e}) = m_{e+1}  \notin N_{e+1}$ by definition of an $F$-sequence.  For $r \in P$, $g_{e}(r \circ m_{e}) = g_{e}(r^\frac{1}{q} m_{e}) = r^\frac{1}{q} m_{e+1} = r^\frac{pq}{q} \circ m_{e+1} = r^p \circ m_{e+1} \in N_{e+1}$.  Hence $P^{[p]} \subseteq (N_{e+1} :_R m_{e+1})$. 

\indent On the other hand, if $s \in (N_{e+1} :_R m_{e+1})$, then certainly $s^p \in (N_{e+1} :_R m_{e+1})$.  We then have $g_{e}^{-1}(s^p \ \circ \ m_{e+1}) = g_{e}^{-1}(s^\frac{p}{qp} m_{e+1}) = g_{e}^{-1}(s^\frac{1}{q} m_{e+1}) = s^\frac{1}{q} m_{e} = s \circ m_{e} \in N_e$, again by definition of $F$-sequences, and $s \in P$.

\indent Finally, since $P^{[p]} \subseteq (N_{e+1} :_R m_{e+1}) \subseteq P$, and $P = \sqrt{P^{[p]}}$, $P$ is minimal over $P^{[p]}$, and hence $P$ is minimal over the annihilator of an element of $F^{e+1}(M)/N_{e+1}$, and since $R$ is Noetherian (see sec.~\ref{sec:GenAss}), $P \in \wAss_R \big( F^{e+1}(M)/  N_{e+1} \big) = \Ass_R \big( F^{e+1}(M)/N_{e+1} \big)$.
\end{proof}

\item  For every $e$, $N_e$ is Frobenius closed.

\begin{proof}
\indent Let $N_e$ be a term in an $F$-sequence, and fix $n_{e} \in N_e^F$.  Then $g_{e,e'}(n_e) = n_{e'} \in N^{[q']} \subseteq N_{e'}$ for $e' \gg e$.  But then $n_e \in g_{e,e'}^{-1}(n_{e'})$, whereby $n_e \in g_{e,e'}^{-1}(N_{e'}) = N_e$.
\end{proof}

\item $\sqrt{\ann_R \big( F^e(M)/N_e \big)} = \sqrt{\ann_R \big( F^{e'}(M)/N_{e'} \big)}$ for all $e, e'$.

\begin{proof}
By lemma~\ref{lem:fseqann}, $\lbrace J_e \rbrace = \lbrace \ann_R \big( F^e(M)/N_e \big) \rbrace$ is an $f$-sequence.  Then this statement is true by remark~\ref{rmk:fseqfacts}, fact $3$.
\end{proof}

\item  There exists a bijective correspondence between $F$-sequences of $F^e(M)$ and submodules of $R^\infty \otimes_R M$.  This fact is true because any sub-module of $N \subseteq R^\infty \otimes_R M$ is by definition the direct limit of an $F$-sequence of modules.

\indent For an submodule $N \subseteq R^\infty \otimes_R M$, this correspondence is given explicitly by $\Gamma \ : \ N \mapsto \lbrace N_e \rbrace$, where for every $e \in \mathbb{N}$, $N_e := \lbrace n_{e} \in F^e(M) \ | \ n_{\infty} \in N \rbrace$. 
\end{enumerate} 
\end{rmk}

\begin{example}\label{example:Fseqs} \hspace{1 cm}
\begin{enumerate}
\item Recall the $f$-sequence of ideals $\lbrace J_e \rbrace = \lbrace (I^{[q]})^F \rbrace$ for some ideal $I$.  Here $N_e = (R^{\frac{1}{f^e}} \otimes_R I)^F \cong (IR^{\frac{1}{p^e}})^F$, where this last expression is isomorphic to $(I^{[q]})^F \subseteq R$.  Notice that $N = \underrightarrow{\lim} \ N_e \cong IR^\infty$, which as we stated above was the ideal in $R^\infty$ corresponding to this $f$-sequence of ideals.
\item For $R =  k[x,y]$, recall the $f$-sequence of ideals $\lbrace J_e \rbrace = \lbrace (x, y^q) \rbrace$.  Here $N_e \cong (x^{\frac{1}{q}},y)R^{\frac{1}{q}}$, which is isomorphic to $(x,y^q) \subseteq R$.  Notice that $N = \underrightarrow{\lim} \ N_e \cong (x, x^{\frac{1}{p}}, x^{\frac{1}{p^2}}, \ldots, y) R^\infty$, which again was the ideal in $R^\infty$ corresponding to this $f$-sequence of ideals. 
\item Letting $R$ and $M$ be as above, then $\lbrace 0^F_{F^e(M)} \rbrace$ is an $F$-sequence of sub-modules of $F^e(M)$.  This fact is easy to see because:
\begin{align*}
m_e \in 0^F_{F^e(M)} &\Leftrightarrow g_{e,e'}(m_e) = 0 \in F^{e'}(M) \ \text{for} \ e' \gg e \\
&\Leftrightarrow g_{e}(m_{e}) \in 0^F_{F^{e+1}(M)}
\end{align*}
As stated above in remark~\ref{rmk:Zero}, $N = 0 = \underrightarrow{\lim} \ 0^F_{F^e(M)}$.
\end{enumerate}   
\end{example}

\subsection{Generalized associated prime ideals}\label{sec:GenAss}

\indent \indent Associated prime ideals of a module become more subtle over a non-Noetherian ring.  If $S$ is a ring, we state four subsets of $\Spec(S)$ which are in general distinct. 

\begin{defn}
Let $S$ be any commutative ring with identity (not necessarily of prime characteristic).  Let $M$ be an $S$-module.  Let $P \in \Spec(S)$ be a prime ideal.

\begin{enumerate}
\item $P \in \Ass_S(M)$ is an \emph{associated prime ideal} of $M$, if $P$ is the annihilator in $S$ of some non-zero element $m \in M$. In such a case we write $P = \text{ann}_S(m)$.  That is, $rm = 0$ for all $r \in P$, and if $r'm = 0$ then $r' \in P$.
\item $P \in  \wAss_{S} (M)$ is a \emph{weakly associated prime}, or weak Bourbaki prime of $M$ if it is minimal over some $\text{ann}_{S} (m)$ for some $m \in M$.  That is, $\text{ann}_{S} (m) \subseteq P$, and if there exists a prime ideal $Q \subset S$ such that $\text{ann}_{S} (m) \subseteq Q \subseteq P$, then $Q = P$.
\item $P \in  \sK_{S} (M)$ is a \emph{strong Krull prime} of $M$ if for any finitely generated sub-ideal $I \subseteq P$ we have $ I \subseteq \text{ann}_{S} (m) \subseteq P$ for some $m \in M$.
\item $P \in  \K_{S} (M)$ is a \emph{Krull prime} of $M$ if for any element $ r \in P$ we have $ r \in \text{ann}_{S} (m) \subseteq P$ for some $m \in M$.
\end{enumerate}
\end{defn}

\indent For any ring $S$ and $S$-module $M$, we always have the containments $\Ass_S(M) \subseteq \wAss_{S} (M) \subseteq \sK_{S} (M) \subseteq \K_{S} (M)$, with none of the containments reversible in general.  For an overview see \cite{IrRu-ass}.  It is possible that $\Ass_S(M)$ is empty, and we provide an example.

\begin{example}
Let $R = \mathbb{F}_p [x]$, and hence $R^\infty  = \mathbb{F}_p [x, x^{\frac{1}{p}}, x^{\frac{1}{p^2}}, \ldots]$.  We claim $\Ass_{R^\infty}(R^\infty / (x) R^\infty) = \emptyset$.  

\indent To show this claim, first note that $(x)^\infty = \sqrt{(x) R^\infty}$.  Since $(x)^\infty$ is maximal it is therefore the only prime ideal which contains $(x)R^\infty$, and hence is the only possible element of $\Ass_{R^\infty}(R^\infty / (x) R^\infty)$.

\indent Suppose $(x)^\infty = ((x) R^\infty :_{R^\infty} r)$ for some $r \in R^\infty \setminus (x)R^\infty$,  where $r = \Sigma_{i = 0}^{n} a_i x^{\frac{i}{q}}$ for some $q$, and $a_i \in \mathbb{F}_p$ for all $i$.  Then $x^{\frac{1}{pq}}r \notin (x) R^\infty$.  We can see this fact by multiplying: 

\begin{align*}
x^{\frac{1}{pq}} r & = x^{\frac{1}{pq}} \Sigma_{i = 0}^{n} a_i x^{\frac{i}{q}} \\
& = \Sigma_{i = 0}^{n} a_i x^{\frac{1}{pq}} x^{\frac{i}{q}} \\
& = \Sigma_{i = 0}^{n} a_i x^{\frac{ip + 1}{qp}}
\end{align*}

\noindent But since $r \notin (x)R^\infty$, for at least one $i$, the corresponding monomial is non-zero, and $i < q$.  But then $\frac{ip + 1}{qp} < 1$, and we see that $x^{\frac{1}{pq}} r \notin (x)R^\infty$.  

\indent Hence $(x)^\infty$ cannot be an associated prime of $R^\infty / (x) R^\infty$, and thus \\ $\Ass_{R^\infty}(R^\infty / (x) R^\infty) = \emptyset$.
\end{example} 
 
 \indent While $\Ass_S(M)$ can be an empty set, if $S$ is a Noetherian ring and $M \neq 0$, there always exists at least one associated prime of $M$.  In this case the first three sets are equal, i.e. $\Ass_S(M) = \wAss_{S} (M) = \sK_{S} (M)$.  For an example of why $\K_{S} (M)$ is omitted from this equality, see \cite[Remark $2.2$]{nmeSh-sKflat}.  However, if $S$ is Noetherian and $M$ is finitely generated, then all four sets are equal, and $\Ass_S(M)$ is a finite set.  
 
\subsection{Generalized depth}\label{sec:GenDepth}

\indent \indent Let $S$ be any ring, again not necessarily Noetherian, let $I \subset S$ be an ideal, and let $M$ be an $S$-module.  The \emph{grade} of $I$ on $M$ is another fundamential notion which generalizes into multiple concepts in the non-Noetherian context.  Discussions of generalized grades can be found in \cite{Ho-grade} and \cite{Fb-ThyGr}, and we include their definitions.

\begin{defn}
Let $S$ be any commutative ring with identity, let $M$ be an $S$-module, and let $I\subset S$ be an ideal such that $M \neq IM$.
\begin{enumerate}
\item $\cgr_S (I,M) := \text{sup} \lbrace | \mbf{x} | \rbrace$ where $\mbf{x} \subset I$ is a finite $M$-regular sequence contained in $I$.
\item $\kgr_S (I,M) := \text{sup} \lbrace n - h \rbrace$ where $\mbf{x} = x_1, \ldots, x_n \subset I$ is a finite set, and $h$ is index of the highest non-zero homology group of the Koszul complex $K_\bullet (x_1, \ldots, x_n ; M)$.  If $\mbf{x}$ is a finite sequence, we can discuss the koszul grade on $\mbf{x}$, $\kgr_S (\mbf{x},M) := n - h$ with $n$ and $h$ as above.
\item $\rgr_S (I,M) := \text{inf} \lbrace i \ | \ \Ext^i_S (S/I, M) \neq 0 \rbrace$.
\end{enumerate} 
\end{defn}
 
\indent For arbitrary $S, I, \text{and} \ M$, $\cgr_S (I,M) \leq \kgr_S (I,M) \leq \rgr_S (I,M)$ is always true, while each inequality can be strict.  See \cite[Remark $2.2$]{nmeSh-sKflat} for an example where $\text{c} \ \gr_S (I,M) = 0 < \kgr_S (I,M)$, in which the inequality holds by Remark~\ref{rmk:AssDepth} below.  See \cite[$2$]{Fb-ThyGr} for an example where $\kgr_S (I,M) < \rgr_S (I,M)$.  

\indent If $I \subseteq J$ are ideals in $S$, then for each notion \underline{\hspace{.3 cm}} grade, we have \underline{\hspace{.3 cm}} $\gr_S (I, M) \leq$ \underline{\hspace{.3 cm}} $\gr_S (J, M)$, \cite{Fb-ThyGr}.  For example, $\cgr_S (I,M) \leq \cgr_S (J,M)$.

\begin{example} 
These measures could be infinite.  Let $S = k[x_1, x_2, \ldots ]$ be the non-Noetherian polynomial ring over some field $k$ with infinitely many variables, let $I = (x_1, x_2, \ldots )$, and $M = S$.  Then $x_1, x_2, \ldots$ is an infinite regular sequence over $M$.  Hence $\cgr_S (I,M)$ is infinite, as are the other two measures since $\cgr_S (I,M) \leq \kgr_S (I,M) \leq \rgr_S (I,M)$.
\end{example}

\indent If $S$ is a Noetherian ring and $M$ is a finitely generated $S$-module, then all three concepts \underline{\hspace{.3 cm}} $\gr_S (I, M)$ coincide, and are indeed finite.  In such a case, all maximal $M$-sequences in $I$ have the same length. 

\indent If $(S, \mathfrak{n})$ is local, the \underline{\hspace{.3 cm}} \emph{depth} is, $\underline{\hspace{.3 cm}} \ \depth_S (M) := \underline{\hspace{.3 cm}} \ \gr_S (\mathfrak{n},M)$.  The relationships between generalized associated prime ideals and generalized depth are an exercise following from their definitions.

\begin{rmk}\label{rmk:AssDepth}
Let $(S, \mathfrak{n}, l)$ be a local ring with quotient field $l = S/\mathfrak{n}$, and let $M$ be an $S$-module such that $M \neq \mathfrak{n} M$. Then,
\begin{enumerate}
\item $\mathfrak{n} \in \Ass_S (M) \ \text{if and only if} \ \rdepth_S (M) = 0$
\item $\mathfrak{n} \in \sK_S (M) \ \text{if and only if} \ \kdepth_S (M) = 0$
\item $\mathfrak{n} \in \K_S (M) \ \text{if and only if} \ \cdepth_S (M) = 0$
\end{enumerate}
\end{rmk}

\begin{proof} 
\indent \underline{1):} 
$$\begin{array}{ll}
\mathfrak{n} \in \Ass_S (M) & \Leftrightarrow 0 \neq \Hom_S(l,M) \cong \Ext_S^0(l,M) \\
& \Leftrightarrow \rdepth_S (M) = 0
\end{array}$$
\indent \underline{2):} 
$$\begin{array}{ll}
\mathfrak{n} \in \sK_S (M) & \Leftrightarrow \text{for all finite} \ \mbf{x} = x_1, \ldots, x_n \subset \mathfrak{n}, (\mbf{x}) \subseteq \text{ann}_S(m) \subseteq \mathfrak{n} \\
& \ \ \ \ \text{for some non-zero} \ m \in M \\
& \Leftrightarrow \ \text{for all finite} \ \mbf{x} = x_1, \ldots, x_n \subset \mathfrak{n}, H_n(\mbf{x};M) \neq 0 \\
& \Leftrightarrow \text{sup} \lbrace \kgr_S(\mbf{x};M) \ | \ \mbf{x} \subset \mathfrak{n} \ \text{is a finite set} \rbrace = 0 \\
& \Leftrightarrow \kdepth_S(M) = 0
\end{array}$$
\indent \underline{3):}
$$\begin{array}{ll}
\mathfrak{n} \in \K_S (M) & \Leftrightarrow \text{for all} \ x \in \mathfrak{n}, \ x \in \text{ann}_S(m) \subseteq \mathfrak{n} \ \text{for some} \ m \in M, \ \text{i.e.,} \\
& \ \ \ \ \text{there exist no} \ M$-$\text{regular elements in} \ \mathfrak{n} \\
& \Leftrightarrow \text{sup} \lbrace |\mbf{x}| \ | \ \mbf{x} \subset \mathfrak{n} \ \text{is a regular sequence} \rbrace = 0 \\
& \Leftrightarrow \cdepth_S(M) = 0 \\
\end{array}$$
\end{proof}

\indent Note that we have no equivalence for $\wAss_S (M)$.  These generalized depths do not relate to weakly associated primes in an obvious way. 

\indent Also note that if $S$ is Noetherian and $M$ is finitely generated, then all concepts coincide, and simply $\mathfrak{n} \in \Ass_S (M) \ \text{if and only if} \ \depth_S (M) = 0$.

\subsection{F-pure rings}\label{sec:F-pure}

\indent \indent Let $S$ be any ring, and let $A$ and $B$ be $S$-modules.  An $S$-module map $\psi: A \rightarrow B$ is \emph{pure} if the map $\psi \otimes_S 1_M: A \otimes_S M \rightarrow B \otimes_S M$ is injective for all $S$-modules $M$.  Alternately we say that $A$ is a pure submodule of $B$.  

\indent We say that a Noetherian ring $R$ of characteristic $p > 0$ is \emph{$F$-pure} if the embedding $\psi: R \hookrightarrow R^{\frac{1}{p}}$ is a pure map.  Equivalently, for any $R$-module $M$ the induced map $M \rightarrow F(M) := R \otimes_R M \rightarrow R^{\frac{1}{p}} \otimes_R M$ is an injective map (all modules act flat).   Note that if $R$ is $F$-pure, then for any $q \leq q'$ the embeddings $\psi_{e,e'} : R^{\frac{1}{q}} \hookrightarrow R^{\frac{1}{q'}}$, as well as the embeddings $\psi_{e,e^\infty} : R^\frac{1}{q} \hookrightarrow R^{\infty}$ are pure maps as well.

\indent If a ring $R$ is $F$-pure, then for any $R$-module $M$ we have that $M \subseteq F(M)$ is a pure submodule as well.  This statement holds since given $M \hookrightarrow F(M) := (R \otimes_R M) \hookrightarrow (R^{\frac{1}{f}} \otimes_R M)$, and given any $R$-module $N$, we have $\big( M \otimes_R N \big)  \rightarrow \big( F(M) \otimes_R N \big)$:

$$\begin{array}{cl}
&= \big( (R \otimes_R M) \otimes_R N \big)  \rightarrow \big( (R^{\frac{1}{f}} \otimes_R M) 
\otimes_R N \big) \\
&= \big( R \otimes_R (M \otimes_R N) \big)  \rightarrow \big( R^{\frac{1}{f}} \otimes_R (M \otimes_R N) \big) \\
&= \big( M \otimes_R N \big) \rightarrow \big( F(M \otimes_R N) \big)
\end{array}$$

\noindent which is an injection by $F$-purity of $R$. 

\begin{rmk}\label{rmk:Fp-Fcl}
If $R$ is an $F$-pure ring, then every ideal $I \subset R$ is Frobenius closed, since if $r^q \in I^{[q]}$, then $r \in IR^{\frac{1}{q}}$, and hence $\overline{r} = \overline{0}$ in $R^{\frac{1}{q}} / IR^{\frac{1}{q}} \cong (R^{\frac{1}{q}} \otimes_R R / I)$.  But since $R$ is $F$-pure, $R/I \cong (R \otimes_R R / I) \rightarrow (R^{\frac{1}{q}} \otimes_R R / I)$ is injective, $\overline{r} = \overline{0}$ in $R/I$, and $r \in I$. 

\indent In particular all $F$-pure rings are reduced, since the $0$ ideal is Frobenius closed.  Hence we need not mention the reduced assumption in section~\ref{sec:F-pureResults}.
\end{rmk}

\begin{rmk}
A particular case of characteristic $p>0$ rings is \emph{$F$-finite} rings $R$, for which $R^{\frac{1}{p}}$ is finitely generated as a right $R$-module.  In such a case $F$-purity is equivalent to \emph{$F$-splitness} \cite[Corollary 5.3]{HoRo-purity}, meaning rings such that the embedding $R \hookrightarrow R^{\frac{1}{p}}$ splits.  We will not discuss this characterization.  However the reader will be advised that under the assumption of $F$-finiteness, $F$-splitness is a sufficient condition in order for all results in section~\ref{sec:F-pureResults} to hold.
\end{rmk}

\begin{example}
\indent Some examples of $F$-pure rings are 

\begin{enumerate}
\item Regular rings such as the polynomial ring $S = k[x_1, \ldots, x_n]$, where $k$ is a field of characteristic $p > 0$. 
\item If $(S, \mathfrak{n})$ is any $F$-finite regular local ring and $R = S / I$ is any quotient, $R$ is $F$-pure if and only if $(I^{[q]} :_R I) \nsubseteq \mathfrak{n}^{[q]}$ for all $q$.  This result is known as \emph{Fedder's criterion}~\cite{Fed83}. 
\item A non-example of an $F$-pure ring is $R := [x,y,z] / (x^p - yz^p)$.  We can see that the ideal $(z)$ is not Frobenius closed.  Certainly $x \notin (z)$, but $x^p = yz^p \in (z)^{[p]}$, and hence $x \in (z)^F \setminus (z)$.  

Alternatively, we can see $x = y^{\frac{1}{p}}z \in (z)R^\infty \cap R = (z)^F$.
\end{enumerate} 
\end{example}

\section{Results for generalized associated prime ideals}\label{sec:GenAssResults}

\indent \indent Given a module $M$ over a regular local ring $R$, Epstein and Shapiro showed a characterization of the strong Krull primes of $R^\infty \otimes_R M$. 

\begin{theorem}\label{thm:Neil}\cite[Corollary 4.9]{nmeSh-sKflat}
Let $R$ be a regular Noetherian ring of characteristic $p>0$.  Let $M$ be any $R$-module.  Then \[
\sK_{R^\infty} (R^\infty \otimes_R M) = \bigcup_{\q \in \Ass_RM} \sK_{R^\infty} (R^\infty / \q R^\infty).
\]
\end{theorem}

\indent Omitting the hypothesis of the regularity of $R$, and for an $F$-sequence $N_e \subseteq F^e(M)$, we find a more general result.  

\begin{theorem}\label{thm:Fseqweakass}
Let $R$ be a reduced Noetherian ring of characteristic $p > 0$, and let $R^\infty$ be its perfect closure.  Let $P$ and $P^\infty$ be corresponding prime ideals in $\Spec(R)$ and $\Spec(R^\infty)$ respectively.  Let $\lbrace N_e \rbrace \subseteq \lbrace F^e(M) \rbrace$ be an $F$-sequence of sub-modules of $F^e(M)$ for an $R$-module $M$.  Let $N = \underrightarrow{\lim} \ N_e$ be the corresponding sub-module of $(R^\infty \otimes_R M)$.  Then the following are all equivalent:  
\begin{enumerate}
\item $P \in \bigcup \Ass_R \big( F^e(M) / N_e \big)$ 
\item $P^\infty \in \wAss_{R^\infty} \big( (R^\infty \otimes_R M)/N \big)$
\item $P^\infty \in \sK_{R^\infty} \big( (R^\infty \otimes_R M)/N \big)$
\end{enumerate}
\end{theorem}

\begin{proof}
\underline{$1 \Rightarrow 2$:} Suppose $P \in \bigcup \Ass_R(F^e(M) / N_e)$ with $ P = \big( N_e :_R m_e \big)$.  I.e. for all $r \in P$, we have $ r \circ m_e =  r^{\frac{1}{q}} m_e \in N_e$.  Equivalently, $ r^{\frac{1}{q}} m_\infty = 0$ in $(R^\infty \otimes_R M)/N$.  Thus $ P R^\infty \subset P^{[\frac{1}{q}]} R^\infty \subseteq ( N :_{R^\infty} m_\infty)$. 

\indent On the other hand, fix some $ s^{\frac{1}{q'}} \in ( N :_{R^\infty} m_\infty )$.  If $q' \leq q$, then $s^{\frac{q}{q'}} \in R$ and $ s^{\frac{q}{q'}} \circ m_e = s^{\frac{1}{q'}} m_e \in N_e$.  Hence  $s^{\frac{q}{q'}} \in P$, and $s^{\frac{1}{q'}} \in P^\infty$. 

\indent  However if $q' > q$, consider the image of $ s \circ m_e = {s}^{\frac{1}{q}}m_e$ under the natural map $g_{e, e'} : F^e(M) \rightarrow F^{e'}(M)$:

\begin{align*}
g_{e,e'} ({s}^{\frac{1}{q}}m_e) &=  {s}^{\frac{1}{q}}m_{e'} \\
&= s'^{\frac{q'}{q}} \circ m_{e'} \\
&=  s'^{(\frac{q'}{q} - 1)}s' \circ m_{e'} \\
&= s'^{(\frac{q'}{q} - 1)} \circ {s}^{\frac{1}{q'}}m_{e'} \\
\end{align*}

\noindent But ${s}^{\frac{1}{q'}}m_{e'} \in N_{e'} = N_{e'}^F$ (see remark~\ref{rmk:Fseqfacts} fact $2$), because $ s^{\frac{1}{q'}} \in  ( N :_{R^\infty} m_\infty )$.  Therefore $g_{e, e'}({s} \ \circ \m_e) \in N_{e'}$, and $s \ \circ \ m_e \in N_e$ because $N_e = N_e^F$.  Thus $ s \in (N_e :_R m_e )$, whereby $s \in P$ and $s^{\frac{1}{q'}} \in P^\infty$.

\indent We have now shown that $ PR^\infty \subseteq ( N :_{R^\infty} m_\infty) \subseteq P^\infty$.  Since $\sqrt{P R^\infty} = P^\infty$, $P^\infty$ is minimal over $P R^\infty$.  Hence $P^\infty$ is minimal over $( N :_{R^\infty} m_\infty)$ as well.  This fact shows that $P^\infty$ is minimal over the annihilator of an element of $(R^\infty \otimes_{R^\infty} M)/N$, and therefore $P^\infty \in \wAss_R \big( (R^{\infty} \otimes_R M)/N \big)$. 

\underline{$2 \Rightarrow 3$:} Always true, as stated above. 

\underline{$3 \Rightarrow 1$:} Suppose $P^\infty \in \sK_R \big( (R^\infty \otimes_R M)/N \big)$, then $PR^\infty$ is a finitely generated sub-ideal, and we have $ PR^\infty \subseteq ( N :_{R^{\infty}} m_\infty ) \subseteq P^\infty$ for some $m_\infty$.  If $P = (r_1, \ldots, r_n)$, then for each $i$ there exists a $q_i$ such that $r_i^{\frac{1}{q_i}}m_{e_i} \in N_{e_i}$.  Letting $q = \text{min} \lbrace q_i \rbrace$, we have $P \subseteq ( N_e :_{R} m_e )$.

\indent On the other hand, fix $ s \in ( N_e :_{R} m_e )$, whereby $s \circ m_e = s^\frac{1}{q}m_e \in N_e$.  Again by definition, $ s^\frac{1}{q} \in ( N :_{R^{\infty}} m_\infty) \subseteq P^\infty$.  Thus $s \in P^\infty \cap R = P$.   Therefore $( N_e :_{R} m_e) \subseteq P$. 

\indent We now have $ P = ( N_e :_{R} m_e)$, and $P \in \Ass_R \big( F^e(M) / N_e \big)$. 
\end{proof}

\begin{rmk}
In the proof we use the fact that $P \in \Ass_R(F^e(M)/N_e)$ under left action.  However, consider the localized module $(F^e(M)/N_e)_P$ over $R_P$, then this module has left depth $0$, i.e. there exist no left-regular elements for this module in $R_P$.  Fix $r \in PR_P$, then $r$ is left-regular over $(F^e(M)/N_e)_P$ if and only if $r^q$ is left-regular if an only if $r$ is right-regular.  The first statement is true by properties of regular sequences (specifically here of length $1$)~\cite[Theorem 5.1.3]{No-FFR}, and the second statement is true because left action by $r^q$ is equivalent to right action by $r$.  Therefore $(F^e(M)/N_e)_P$ over $R_P$ also has no right-regular elements, and $\Ass_R \big( F^e(M) / N_e \big)$ is a two-sided concept.
\end{rmk}

\indent  As stated in example~\ref{example:Fseqs}, $\lbrace 0^F_{F^e(M)} \rbrace$ is an $F$-sequence.  We therefore have a particular case of theorem~\ref{thm:Fseqweakass} which more directly addresses theorem~\ref{thm:Neil}.

\begin{cor}\label{cor:WeakAss}
Let $R$ be a reduced Noetherian ring of characteristic $p > 0$, and let $R^\infty$ be its perfect closure.  Let $P$ and $P^\infty$ be corresponding prime ideals in $\Spec(R)$ and $\Spec(R^\infty)$ respectively.  Let $M$ be an $R$-module.  Then the following are all equivalent:  
\begin{enumerate}
\item $P \in \bigcup \Ass_R \big( F^e(M) / 0^F_{F^e(M)} \big)$
\item $P^\infty \in \wAss_{R^\infty}(R^{\infty} \otimes_R M)$
\item $P^\infty \in \sK_{R^\infty}(R^{\infty} \otimes_R M)$
\end{enumerate}
\end{cor}

\begin{proof}
\indent Let $\lbrace N_e \rbrace = \lbrace 0^F_{F^e(M)} \rbrace$, and $N = 0 =  \underrightarrow{\lim} \ 0^F_{F^e(M)}$.  The result follows from theorem~\ref{thm:Fseqweakass}.
\end{proof}

\indent Another particular case of theorem~\ref{thm:Fseqweakass} is when $\lbrace J_e \rbrace$ is an $f$-sequence of ideals in a Noetherian ring. 

\begin{cor}\label{thm:CycModAssocPr}
Let $R$ be a reduced Noetherian ring of characteristic $p > 0$, and let $R^\infty$ be its perfect closure.  Let  $\lbrace J_e \rbrace$ be an $f$-sequence in $R$, and let $J \subset R^\infty$ be its corresponding ideal in $R^\infty$. Let $P$ and $P^\infty$ be corresponding prime ideals in $\Spec(R)$ and $\Spec(R^\infty)$ respectively.  Then the following are all equivalent:  

\begin{enumerate}
\item $P \in \bigcup \Ass_R (R / J_e)$
\item $P^\infty \in \wAss_{R^\infty}(R^{\infty} / J)$
\item $P^\infty \in \sK_{R^\infty}(R^{\infty} / J$)
\end{enumerate} 
\end{cor}

\begin{proof}
Letting $M = R$, $\lbrace N_e \rbrace = \lbrace J_e \rbrace$, and $N = J$ we see that this corollary is a special case of theorem \ref{thm:Fseqweakass}.
\end{proof}

\indent Lastly for an ideal $I \subseteq R$, we have the particular case for the minimal $f$-sequence $\lbrace (I^{[q]})^F \rbrace$.

\begin{cor} Let $R$ be a reduced Noetherian ring of characteristic $p > 0$, let $I \subset R$ be an ideal, and let $R^\infty$ be its perfect closure.  Let $P$ and $P^\infty$ be corresponding prime ideals in $\Spec(R)$ and $\Spec(R^\infty)$ respectively.  Then the following are all equivalent:  
\begin{enumerate}
\item $P \in \bigcup \Ass_R (R / (I^{[q]})^F)$
\item $P^\infty \in \wAss_{R^\infty}(R^{\infty} / I R^{\infty})$
\item $P^\infty \in \sK_{R^\infty}(R^{\infty} / I R^{\infty})$
\end{enumerate}
\end{cor}

\begin{proof}
\indent Letting $J = I R^\infty$, and $J_e = (I^{[q]})^F$ for all $e$, we see this corollary is a special case of corollary~\ref{thm:CycModAssocPr}.

\indent Alternately, letting $N_e = (I^{[q]})^F$ and $N = IR^\infty$, we see the corollary is a special case of theorem~\ref{thm:Fseqweakass}.
\end{proof}

\section{Results over $F$-pure rings}\label{sec:F-pureResults}

\subsection{Stabilizing depth}\label{sec:SDepth}

\indent \indent In an $F$-pure ring, since all ideals $I \subseteq R$ are Frobenius closed (see Remark~\ref{rmk:Fp-Fcl}), $\lbrace (I^{[q]})^F \rbrace = \lbrace I^{[q]} \rbrace$ is always an $f$-sequence.  As stated in remark~\ref{rmk:fseqfacts}, $\Ass_R(R / I^{[q]})$ can  gain elements as $q$ grows.  We ask: is depth$_R(R/I^{[q]})$ non-increasing as $q$ grows?  Or considering that

\begin{align*}
R / I^{[p]} &\cong R^{\frac{1}{q}} / I R^{\frac{1}{q}} \\
&\cong (R^{\frac{1}{p}} \otimes_R R / I) \\
&= F(R / I)
\end{align*} 

\noindent we can more generally ask: is depth$_R\big(F^e(M)\big)$ non-increasing as $e$ grows?

\indent We must first state a lemma addressing the differing left and right actions of $R$ on $F(M)$.

\begin{lemma}\label{lem:twosidereg}
Let $R$ be an $F$-pure Noetherian ring of characteristic $p > 0$.  Let $M$ be an $R$-module, and let $\mbf{x} = x_1, \ldots, x_n \subset R$ be a finite sequence.  Then for any $e \in \mathbb{N}$, $\mbf{x}$ is a regular left  $F^e(M)$-sequence if and only if $\mbf{x}$ is a regular right $F^e(M)$-sequence.
\end{lemma}

\begin{proof} 
\indent Fix $e \in \mathbb{N}$.  It is easy to see that $\mbf{x}$ is a left-regular $F^e(M)$-sequence if and only if $\mbf{x}^{[q]} = x_1^q, \ldots, x_n^q$ is a left-regular $F^e(M)$-sequence if and only if $\mbf{x}$ is a right-regular $F^e(M)$-sequence. 

\indent The first equivalence is true by properties of regular sequences \cite[Theorem 5.1.3]{No-FFR}, whereby $\mbf{x}$ is a left-regular $F^e(M)$ if and only if $\mbf{x}^{[q]}$ is a left-regular $F^e(M)$ sequence.  

\indent The second equivalence is true because for any $m_e \in F^e(M)$, the left action by any $x_j^q$ for $j = 1, \ldots, n$ is equivalent to right action by $x_j$.  That is,

$$\begin{array}{cl}
x_j^q \circ m_e  &=  x_j m_e \\
&= m_e x_j \\
&= m_e \cdot x_j
\end{array}$$
\end{proof}

\indent We then have a corollary.

\begin{lemma}\label{lem:twosidedepth}
Let $(R,\mathfrak{m})$ be an $F$-pure Noetherian local ring of characteristic $p > 0$.  Let $M$ be an $R$-module, and let $F^e(M)$ denote the image of $M$ under the $e$-th iteration of the Frobenius functor for $e \in \mathbb{N}$.  Then $\leftdepth_R F^e(M) = \rightdepth_R F^e(M)$.  Two-sided $\depth_R F^e(M)$ is therefore a well-defined concept.
\end{lemma}

\begin{proof}
\indent By lemma~\ref{lem:twosidereg}, $\mbf{x} \subset R$ is left $F^e(M)$-regular if and only if it is right $F^e(M)$-regular.  Therefore all maximal such sequences have the same length under right and left action.
\end{proof}

\indent Our next lemma is a fact regarding all pure submodule inclusions.

\begin{lemma}\label{lem:puresubreg}
Let $(R, \mathfrak{m})$ be a reduced Noetherian local ring.  Let $M \subseteq N$ be a pure inclusion of $R$-modules such that $M \neq 0$ is finitely generated over $R$.  Let the finite set $\mbf{x} \subset R$ be an $N$-regular sequence.  Then $\mbf{x}$ is regular over $M$ as well.
\end{lemma}

\begin{proof} 
We proceed by induction on $|\mbf{x}|$.  Let $\mbf{x} = x$ be a singleton.  Then $M / xM \neq 0$ by Nakayama's lemma since $x \in \mathfrak{m}$ and $M \neq 0$. Furthermore, $x$ clearly annihilates no element of $M \subset N$.  

\indent Now assume the statment is true for $|\mbf{x}| = n - 1$.  Let $|\mbf{x}| = n$, and let $\mbf{x}' = x_1, \ldots, x_{n-1}$ denote the first $n-1$ elements of $\mbf{x}$.  Again, $M / \mbf{x}' M \neq 0$ by Nakayama's lemma since $\mbf{x}' \subset \mathfrak{m}$ and $M \neq 0$.  

Additionally, $M / \mbf{x}' M$ embeds into $N / \mbf{x}' N$.  In order to see this fact, recall that $M / \mbf{x}' M \cong M \otimes_R R / (\mbf{x}')$, and $N / \mbf{x}' N \cong N \otimes_R R / (\mbf{x}')$, and since $M$ is a pure sub $N$-module we have the embedding $M \otimes_R R / (\mbf{x}') \hookrightarrow N \otimes_R R / (\mbf{x}')$.

Thus since $x_n$ annihilates no element of $N / \mbf{x}' N$, it will therefore annihilate no element of the submodule $M / \mbf{x}' M$.  Therefore, $\mbf{x}$ is an $M$-regular sequence.
\end{proof}

\indent In this context, we now have an answer to our question above.

\begin{lemma}\label{lem:submodreg}
Let $(R, \mathfrak{m})$ be a local Noetherian $F$-pure ring of characteristic $p > 0$, and let $M$ be a finitely generated $R$-module.  If $\mbf{x} = x_1, \ldots, x_n$ is regular over $F(M)$, then it is regular over $M$.
\end{lemma}

\begin{proof}
\indent $M \subseteq F(M)$ is a pure submodule.  Hence by~\ref{lem:puresubreg}, if $\mbf{x}$ is regular over $F(M)$, it is also regular over $M$.
\end{proof}

\indent So for an $F$-pure ring, we do have non-increasing depth under the Frobenius functor.

\begin{theorem}
Let $(R, \mathfrak{m})$ be a local Noetherian $F$-pure ring of characteristic $p > 0$, and let $M$ be a finitely generated $R$-module.  Then,

\begin{center}
 $\depth_R (M) \geq \depth_R \big( F(M) \big)$.
\end{center}
\end{theorem}

\begin{proof}
\indent If $\depth_R \big( F(M) \big) = d$ and $x_1, \ldots, x_d$ is a maximal $F(M)$-sequence, then by lemma~\ref{lem:submodreg}  $x_1, \ldots, x_d$ is $M$-regular, and $\depth_R (M) \geq d$.
\end{proof}

\indent Recall that $F^e(M)$ is always finitely generated as left $R$-module (see remark~\ref{rmk:freakout}).  The identical argument shows $\depth_R \big( F^e(M) \big) \geq \depth_R \big( F^{e+1}(M) \big)$ for every $e$.  I.e., $\depth_R \big( F^e(M) \big)$ is non-increasing $e$ grows.  Since it is bounded below by $0$, the value must stabilize, meaning there must exist some $d \in \mathbb{N}$ for which $\depth_R \big( F^e(M) \big) = d$ for $e \gg 0$.  This fact yields a definition.

\begin{defn}
Let $(R,\mathfrak{m})$ be an F-pure Noetherian local ring of characteristic $p > 0$, and let $M$ be a finitely generated $R$-module.  The \emph{stabilizing depth} is defined:

\begin{center}
$\sdepth_R (M) := \depth_R \big( F^e(M) \big) \ \text{for } \ e \gg 0$.
\end{center}
\end{defn}

\begin{example}\label{example:DecrDepth} For any $t \geq 0$, we can construct a class of examples of modules $M$ for which $\sdepth_R (M) > \sdepth_R (M) = t$.  

\indent Let $S$ be an $F$-pure Noetherian ring of characteristic $p > 0$ and $J \subset S$ be an ideal such that $\Ass_S(S/J^{[q]})$ grows as $q$ grows.  For an example for such cyclic modules, see \cite{SiSwUAss}.  Let $P \in \Ass_S (S / J^{[q]}) \setminus \Ass_S (S / J)$ for some $q$, and let $A = S_P$ be $S$ localized at $P$.  Letting $\mathfrak{a} = PA$ we have a local ring $(A, \mathfrak{a})$.  Now $\depth_{A} (A / J A) > 0 = \sdepth_{A} (A / J A) = \depth_{A} (A / J^{[q]} A)$.

\indent Now let $R := A \llbracket x_1, \ldots, x_t \rrbracket$ for variables $x_i$, which is local with maximal ideal $\mathfrak{m} = (x_1, \ldots, x_t) + \mathfrak{a}R$ \cite[Discussion on Page 4]{Mats}.  Let $I = JR$.  Then for all $q$, $\depth_R(R / I^{[q]}) = \depth_A(A / J^{[q]}) \ + \ t$ (follows from \cite[1.2.16]{BH}).  Hence $\depth_R(R / I) = \depth_A(A / J) + t > t$, but $\depth_R(R / I^{[q]}) = \depth_A(A / J^{[q]}) + t = 0 + t = t$ for $q \gg 0$.  Thus $\depth_R(R / I) > \sdepth_R(R / I) = t$.
\end{example}

\subsection{Depth type comparisons}\label{sec:DepthTypeComp}

\indent \indent We investigate how the stabilizing depth of a finitely generated module $M$ over an $F$-pure ring $R$ relates to the generalized depths of $(R^\infty \otimes_R M)$ over the usually non-Noetherian extension $R^\infty$.  We first compare $\sdepth_R (M)$ with $\cdepth_{R^\infty} (R^\infty \otimes_R M)$, but we begin with a lemma regarding regular sequences from $R$.

\begin{lemma}\label{lem:infreg}
Let $R$ be an $F$-pure Noetherian ring of characteristic $p > 0$, and let $R^\infty$ be its perfect closure.  Let $M$ be an $R$-module, and let $\mbf{x} = x_1, \ldots, x_n \subset R$ be a finite sequence.  Then, $\mbf{x}$ is a regular $F^e(M)$-sequence for all $e \in \mathbb{N}$ if and only if $\mbf{x}$ is a regular $(R^\infty \otimes_R M)$-sequence.
\end{lemma}

\begin{proof} 
\indent If $\mbf{x}$ is a regular $(R^\infty \otimes_R M)$-sequence, then by~\ref{lem:puresubreg}, $\mbf{x}$ is a regular right $F^e(M)$-sequence since $F^e(M) \hookrightarrow (R^\infty \otimes_R M)$ is a pure submodule by $F$-purity of $R$, and $\mbf{x}$ is left-regular if and only it is right-regular. 

\indent Conversely, we proceed by induction on $|\mbf{x}|$.  If $\mbf{x} = x$ is a singleton, suppose $x$ is not $(R^\infty \otimes_R M)$-regular, with $m_\infty x = 0$.  Then $0 = m_{e} x$ in some $F^{e}(M)$, and $0 \neq m_{e}$, since it was not zero in $(R^\infty \otimes_R M) = \underrightarrow{Lim} \ F^e(M)$.

\indent Now suppose $\mbf{x} = x_1, \ldots, x_n \subset R$ is $F^e(M)$-regular for all $e$, in particular right-regular, and let $\mbf{x}' = x_1, \ldots, x_{n-1}$, which is by hypothesis regular over $(R^\infty \otimes_R M)$.  We must show that $x_n$ is regular over

\begin{align*}
(R^\infty \otimes_R M) / (R^\infty \otimes_R M)\mbf{x}' & \cong (R^\infty \otimes_R M) \otimes_R R / \mbf{x}'R \\
& \cong R^\infty \otimes_R ( M \otimes_R R / \mbf{x}'R ) \\
& \cong R^\infty \otimes_R M / \mbf{x}'M
\end{align*}
 
\noindent We then apply the identical argument to the regular element $x_n$ over the module $F^e(M / \mbf{x}'M)$, which is therefore is also regular over the module $(R^\infty \otimes_R M / \mbf{x}'M)$.
\end{proof}

\indent  If $\bf x$ is a maximal regular sequence over $F^e(M)$ for some specific $e$, and $\bf y$ is a maximal regular sequence over $F^e(M)$ for all $e$, clearly $|\mbf{x}| \geq |\mbf{y}|$.

\begin{theorem}
Let $R$ be an $F$-pure Noetherian ring of characteristic $p > 0$, let $R^\infty$ be its perfect closure, and let $M$ be a finitely generated $R$-module.  Then $\cdepth_{R^\infty} (R^\infty \otimes_R M)$ is finite, and more specifically

\begin{center}
$\sdepth_R (M) \geq \cdepth_{R^\infty} (R^\infty \otimes_R M)$
\end{center}

\end{theorem}

\begin{proof}
Let $\bf x'$ $= x_1^{\frac{1}{q_1}}, \ldots, x_n^{\frac{1}{q_n}} \subset R^\infty$ be an $(R^\infty \otimes_R M)$-regular sequence.  Then by properties of regular sequences  \cite[Theorem 5.1.3]{No-FFR}, $\bf x$ $:= x_1, \ldots, x_n \subset R \subseteq R^\infty$ is $(R^\infty \otimes_R M)$-regular as well.  By lemma~\ref{lem:infreg}, $\bf x$ is $F^e(M) $-regular for all $e$.  Hence any maximal $(R^\infty \otimes_R M)$-sequence can have at most the length of $\sdepth_R (M) = \depth_R \big( F^e(M) \big)$ for $e \gg 0$.
\end{proof}

\indent  After c depth, the next largest depth value of $(R^\infty \otimes_R M)$ over $R^\infty$ is the Koszul depth.  We find that it coincides with the stabilizing depth. 

\begin{ntn}
Before we discuss the next result, we make note of some notation used throughout its proof.  If $\mbf{y} \subset R^{\frac{1}{q}} \subseteq R^\infty$ is some finite sequence for some $q$ (perhaps even $\mbf{y} \subset R$), we can construct the koszul complex $K_{\bullet}(\mbf{y} ; F^{e'}(M))$ for any $e' \geq e$, since $R^{\frac{1}{q}} \subseteq R^{\frac{1}{q'}}$.  But we can also construct  $K_{\bullet}(\mbf{y} ; R^\infty \otimes_R M)$.  Since both such complexes are built from the same finite sequence $\mbf{y}$, the matrices defining the differentials $d_j$ are identical in each version.  When we discuss the differentials for the Koszul complex over $F^e(M)$ we write $d_j^q$, while the differentials over $R^\infty \otimes_R M$ are denoted $d_j^\infty$. 

\end{ntn}

\indent Lastly we require a lemma.

\begin{lemma}\label{lem:supkosz}
Let $(R,\mathfrak{m})$ be an F-pure Noetherian local ring of characteristic $p > 0$, let $(R^\infty, \mathfrak{m}^\infty)$ be its perfect closure, and let $M$ be a finitely generated $R$-module.  Then,

\begin{center}
$\kdepth_{R^\infty} (R^\infty \otimes_R M) = \ \sup \lbrace \kgr_{R^\infty} \big(\mathfrak{m}^{[\frac{1}{q}]}R^\infty, (R^\infty \otimes_R M) \big) \rbrace$ for all $e \in \mathbb{N}$.
\end{center}

\end{lemma}

\begin{proof}
\indent We know that $\kdepth_{R^\infty} (R^\infty \otimes_R M) =$ sup$\lbrace \kgr_{R^\infty} \big(I, (R^\infty \otimes_R M) \big) \rbrace$, where $I \subset R^\infty$ is \emph{any} finitely generated sub-ideal.  We therefore must show that the ideals $\mathfrak{m}^{[\frac{1}{q}]}R^\infty$ determine the Koszul depth. 

\indent Recall for ideals $I \subseteq J \subset R^\infty$, we have $\kgr_{R^\infty} (I,M) \leq \kgr_{R^\infty} (J,M)$.  

\indent Let $I = (y_1^{\frac{1}{q_1}}, \ldots, y_n^{\frac{1}{q_n}}) \subset R^\infty$ be a finitely generated ideal.  Then $I \subseteq \mathfrak{m}^\infty$ since $(R^\infty,\mathfrak{m}^\infty)$ is a local ring.  Letting $q \geq \text{max} \lbrace q_i \rbrace$, we see that $I \subseteq \mathfrak{m}^{[\frac{1}{q}]}R^\infty$.  Hence for \emph{any} finitely generated $I \subset R^\infty$, we have $\kgr_{R^\infty} (I,M) \leq \kgr_{R^\infty} (\mathfrak{m}^{[\frac{1}{q}]}R^\infty,M) \big)$ for $q \gg 0$, since it is a sub-ideal.  Therefore, ideals of the form $\mathfrak{m}^{[\frac{1}{q}]}R^\infty$ achieve sup$\lbrace \kgr_{R^\infty} \big(I, (R^\infty \otimes_R M) \big) \rbrace$. 
\end{proof}

\begin{theorem}
Let $(R,\mathfrak{m})$ be an F-pure Noetherian local ring of characteristic $p > 0$, let $R^\infty$ be its perfect closure, and let $M$ be a finitely generated $R$-module.  Then $\kdepth_{R^\infty} (R^\infty \otimes_R M)$ is finite, and more specifically

\begin{center}
$\kdepth_{R^\infty} (R^\infty \otimes_R M) = \sdepth_R (M)$
\end{center}
\end{theorem} 

\begin{proof}
To show $\kdepth_{R^\infty} (R^\infty \otimes_R M) \geq \sdepth_R (M)$, let $\mbf{x} = x_1, \ldots, x_n$ be minimal system of generators for $\mathfrak{m}$.  We claim that we have the comparison $\kgr_{R^\infty} \big((\mbf{x})R^\infty, (R^\infty \otimes_R M) \big) \geq \sdepth_R (M)$, which will prove the inequality, since $\kdepth_{R^\infty} (R^\infty \otimes_R M) = \kgr_{R^\infty} \big(\mathfrak{m}^\infty, (R^\infty \otimes_R M) \big) \geq \kgr_{R^\infty} \big((\mbf{x})R^\infty, (R^\infty \otimes_R M) \big)$.

\indent To prove this claim, let $\depth_R \big( F^e(M) \big) \geq d$ for all $e$, and let $j > n - d$.  Fix $e$.  Recall that since $R$ is Noetherian, $\kdepth_R (F^e(M)) = \depth_R (F^e(M))$.  Thus $H_j \big( \mbf{x}; F^e (M) \big) = 0$, and we must also show $H_j \big( \mbf{x}; (R^\infty \otimes_R M) \big) = 0$.  Fix $z_\infty \in \text{ker}(d_j^{\infty})$.  Then $z_e \in \text{ker}(d_j^{q})$.  By exactness at this $j$-th homology group, there exists a non-zero $y_e \in \big( F^e(M) \big)^{n \choose j+1}$ such that $d_{j+1}^q(y_e) = z_e$.  By $F$-purity of $R$, we have an injection $F^e(M)^{n \choose j+1} \hookrightarrow (R^\infty \otimes_R M)^{n \choose j+1}$, hence the image $y_\infty$ in $(R^\infty \otimes_R M)^{n \choose j+1}$ is non-zero.  But now $d_{j+1}^{\infty}(y_\infty) = z_{\infty}$. 

\indent Since $z_\infty \in \text{ker}(d_j^{\infty})$ was chosen arbitrarily, therefore $\text{ker}(d_j^{\infty}) = \text{image}(d_{j+1}^{\infty})$, and $H_j \big( \mbf{x}; (R^\infty \otimes_R M) \big) = 0$.

\indent To show the reverse inequality, again let $\mbf{x} = x_1, \ldots, x_n$ be minimal system of generators for $\mathfrak{m}$, and recall  by lemma~\ref{lem:supkosz}, $\kdepth_{R^\infty} (R^\infty \otimes_R M) = \sup \lbrace \kgr_{R^\infty} \big((\mbf{x})^{[\frac{1}{q}]}R^\infty, (R^\infty \otimes_R M) \big) \rbrace$.  We claim that for any $e$, $\kgr_{R^\infty} \big((\mbf{x})^{[\frac{1}{q}]}R^\infty, (R^\infty \otimes_R M) \big) \leq \depth_R \big( F^e(M) \big)$. Hence we know $\kdepth_{R^\infty} (R^\infty \otimes_R M)$ is finite, and since the statement is true for $e \gg 0$, the inequality is true.

\indent To prove this claim, fix $e$.  Let $\kgr_{R^\infty} \big((\mbf{x})^{[\frac{1}{q}]}R^\infty, (R^\infty \otimes_R M) \big) = d$, and $j > n - d$, whereby $H_j \big( \mbf{x}^{[\frac{1}{q}]}; (R^\infty \otimes_R M) \big) = 0$.  We must show that $H_j \big( \mbf{x}; F^e(M) \big) = 0$ as well, where $\mbf{x}$ acts on the left, which is equivalent to right action by $\mbf{x}^{\frac{1}{q}} \in R^{\frac{1}{q}}$.  Therefore, we can equivalently ask whether $H_j \big( \mbf{x}^{[\frac{1}{q}]}; F^e(M) \big) = 0$ via right action.  Fix $z_e \in \text{ker}(d_j^q)$, whereby $z_\infty \in \text{ker}(d_j^\infty)$.  Then $\overline{z_\infty} = 0 \in \text{coker}(d_{j+1}^\infty)$ by exactness of $H_{\bullet} \big( \mbf{x}^{\frac{1}{q}}; (R^\infty \otimes_R M) \big)$ at the $j$-th position.  But $\text{coker}(d_{j+1}^\infty) = \text{coker}(d_{j+1}^q \otimes_R 1_{R^\infty}) = \text{coker}(d_{j+1}^q) \otimes_R R^\infty$, and since $R$ is $F$-pure we have an injection $\text{coker}(d_{j+1}^q) \hookrightarrow \text{coker}(d_{j+1}^q) \otimes_R R^\infty$.  Since $\overline{z_{\infty}} = 0$ in the image of this map, therefore its preimage $\overline{z_e} = 0$ as well.

\indent Since $z_e \in \text{ker}(d_j^q)$ was chosen arbitrarily, $H_{\bullet} \big( \mbf{x}; F^e (M) \big)$ is exact at the $j$-th position.
\end{proof}

\indent Combining these two results, for $R$ and $M$ as above:

\begin{center}
$\kdepth_{R^\infty} (R^\infty \otimes_R M) = \sdepth_R (M) \geq \cdepth_{R^\infty} (R^\infty \otimes_R M)$
\end{center}

\noindent In order to find a condition under which this last inequality is an equality, currently we require a further supposition on our rings.  Namely, we require rings which either satisfy countable prime avoidance, or a specific application of the prime avoidance lemma is sufficient.

\begin{defn}
A Noetherian ring $R$ satisfies the countable prime avoidance lemma if for any ideal $I \subset R$, and countable collection of prime ideals $\lbrace P_e \rbrace$, if $I \subseteq \bigcup_{e \in \mathbb{N}} P_e$, then $I \subseteq P_e$ for some $e$.  
\end{defn}

\begin{rmk}
Countable prime avoidance is satisfied by any complete Noetherian local ring, or any ring which contains uncountably many elements $\lbrace u_\lambda \rbrace_{\lambda \in \Lambda}$ for which $u_\lambda - u_\mu$ is a unit for $\lambda \neq \mu$.  In particular, any ring which contains an uncountable field is an example of this second condition (see \cite[Lemma 13.2]{LeWe-CMR}, \cite[Lemma 3]{BurchCPA}, \cite{RsPv-BCT}, \cite{HHexponent}).
\end{rmk}

\begin{theorem}\label{thm:twohyps}
 Let $(R,\mathfrak{m})$ be an F-pure Noetherian local ring of characteristic $p > 0$, and let $M$ be a finitely generated $R$-module.  If either of the following two conditions hold:

\begin{enumerate}
\item $R$ satisfies countable prime avoidance
\item $\ds \bigcup \bigcup_{e} \Ass_R \big( F^e(M) / F^e(M) \mbf{y} \big)$ contains finitely many prime ideals, where $\bf y$ $\subset R$ is a maximal $(R^\infty \otimes_R M)$ sequence in $R^\infty$
\end{enumerate}

then,

\begin{center}
$\sdepth_R (M) = \cdepth_{R^\infty} (R^\infty \otimes_R M)$
\end{center}
\end{theorem}

\begin{proof}
\indent We have already shown that $\sdepth_R (M) \geq \cdepth_{R^\infty} (R^\infty \otimes_R M)$.  

\indent Conversely, suppose that $\sdepth_R (M) > \cdepth_{R^\infty} (R^\infty \otimes_R M) = d$.  Let $\mbf{y} = y_1, \ldots, y_d \subset R$ be a maximal $(R^\infty \otimes_R M)$ sequence.  Then for all $e$, $\depth_R \big( F^e(M) / F^e(M)\mbf{y} \big) > 0 = \cdepth_{R^\infty} \big( (R^\infty \otimes_R M) / (R^\infty \otimes_R M) \mbf{y} \big)$.  We claim that $\mathfrak{m} = \ds \bigcup \bigcup_{e} \Ass_R \big( F^e(M) / F^e(M) \mbf{y} \big)$.  Note that this union consists of at most countably many prime ideals since it is a countable union of finite sets. 

\indent Clearly the union is contained in $\mathfrak{m}$, since $(R, \mathfrak{m})$ is local.  Conversely, suppose that there exists some $z \in \mathfrak{m} \setminus \ds \bigcup \bigcup_{e} \Ass_R \big( F^e(M) / F^e(M) \mbf{y} \big)$.  Then $z$ is not contained in any associated prime of $F^e(M) / F^e(M) \mbf{y}$ for any $e$.  Hence $z$ is regular over $F^e(M) / F^e(M) \mbf{y}$ for all $e$.  Then $\lbrace \mbf{y},z \rbrace$ is a right regular $F^e(M)$-sequence for all $e$, and hence it is $(R^\infty \otimes_R M)$-regular by lemma~\ref{lem:infreg}.  But this fact contradicts the maximality of $\mbf{y}$. 

\indent Now suppose either condition of the theorem holds.  Since $\mathfrak{m}$ is contained in this union of prime ideals, then $\mathfrak{m} \subseteq P$ for some $P \in \Ass_R \big( F^e(M) / F^e(M) \mbf{y} \big)$ and for some $e$.  But since $\mathfrak{m}$ is maximal, $\mathfrak{m} = P$.  We now have the maximal ideal of $R$ shown to be an associated prime of $F^e(M) / F^e(M) \mbf{y}$, and hence $\depth_R \big( F^e(M) / F^e(M) \mbf{y} \big) = 0$.  We have contradicted our assumption that $\depth_R \big( F^e(M) / F^e(M)\mbf{y} \big) > 0$.
\end{proof}

\section*{Acknowledgments}

\indent The author would like to thank his thesis advisor Neil Epstein for his guidance throughout this project.  And though our formal mentorship is over, he has still been incredibly generous with his time while proofreading rough drafts of this paper.  In this latest role, he can particularly be thanked for discussions of $F$-sequences, and the suggestion of lemma~\ref{lem:fseqann}.

\bibliographystyle{amsalpha}
\bibliography{rsch_refs}

\end{document}